\documentclass[12pt, reqno]{amsart}
\usepackage{ amsmath,amsthm, amscd, amsfonts, amssymb, graphicx, color}
\usepackage[bookmarksnumbered, colorlinks, plainpages]{hyperref}
\newtheorem{thm}{Theorem}[section]
\newtheorem{cor}[thm]{Corollary}
\newtheorem{lem}[thm]{Lemma}
\newtheorem{prop}[thm]{Proposition}

\newtheorem{exam}[thm]{Example}
\numberwithin{equation}{section}

\begin{document}

\title{On Yaqub nil-clean rings}

\author{Huanyin Chen}
\author{Marjan Sheibani Abdolyousefi}
\address{
Department of Mathematics\\ Hangzhou Normal University\\ Hang -zhou, China}
\email{<huanyinchen@aliyun.com>}
\address{
Farzanegan Semnan University\\ Semnan, Iran}
\email{<sheibani@fgusem.ac.ir>}

\subjclass[2010]{16U99, 16E50, 13B99.} \keywords{Nilpotent; tripotent; exchange ring; Yaqub nil-clean ring}

\begin{abstract}
A ring $R$ is a Yaqub nil-clean if $a+a^3$ or $a-a^3$ is nilpotent for all $a\in R$. We prove that a ring $R$ is a Yaqub nil-clean ring if and only if $R\cong R_1, R_2,R_3,R_1\times R_2$ or $R_1\times R_3$, where
$R_1/J(R_1)$ is Boolean, $R_2/J(R_2)$ is a Yaqub ring, $R_3/J(R_3)\cong {\Bbb Z}_5$ and each $J(R_i)$ is nil, if and only if $J(R)$ is nil and
$R/J(R)$ is isomorphic to a Boolean ring $R_1$, a Yaqub ring $R_2$, ${\Bbb Z}_5$, $R_1\times R_2$, or $R_1\times {\Bbb Z}_5$, if and only if for any $a\in R$, there exists $e^3=e$ such that $a-e$ or $a+3e$ is nilpotent and $ae=ea$, if and only if $R$ is an exchange Hirano ring. The structure of such rings is
thereby completely determined.
\end{abstract}

\maketitle

\section{Introduction}
Throughout, all rings are associative with an identity. A ring $R$ is strongly nil-clean if $a-a^2$ is nilpotent for all $a\in R$.
A ring $R$ is strongly weakly nil-clean if $a+a^2$ or $a-a^2$ is nilpotent for all $a\in R$.
Strongly (weakly) nil-clean rings are studied by many authors, e.g., ~\cite{BDZ, BGD}, \cite{CS2}, \cite{Da1,D} and \cite{KWZ, KWZ2}. An element $a$ in a ring is tripotent if $a^3=a$. A ring is strongly 2-nil-clean if $a-a^3$ is nilpotent for all $a\in R$ (see~\cite{CS}).
It is proved that a ring $R$ is strongly 2-nil-clean if for any $a\in R$ there exists a tripotent $e\in R$ such that $a-e\in R$ is nilpotent and
$ae=ea$ (see~\cite[Theorem 2.8]{CS}).

We call a ring $R$ is Yaqub nil-clean if $a+a^3$ or $a-a^3$ is nilpotent for all $a\in R$. Clearly, strongly weakly nil-clean and strongly 2-nil-clean rings are Yaqub nil-clean, but the converse is not true, e.g., ${\Bbb Z}_5$. The motivation is to determine the structure of such rings.

A ring $R$ is a Yaqub ring provided that it is the subdirect product of ${\Bbb Z}_3$'s. (see~\cite{CS}).
We prove that a ring $R$ is a Yaqub nil-clean ring if and only if $R\cong R_1, R_2,R_3,R_1\times R_2$ or $R_1\times R_3$, where
$R_1/J(R_1)$ is Boolean, $R_2/J(R_2)$ is a Yaqub ring, $R_3/J(R_3)\cong {\Bbb Z}_5$ and each $J(R_i)$ is nil, if and only if $J(R)$ is nil and
$R/J(R)$ is isomorphic to a Boolean ring $R_1$, a Yaqub ring $R_2$, ${\Bbb Z}_5$, $R_1\times R_2$, or $R_1\times {\Bbb Z}_5$, if and only if for any $a\in R$, there exists $e^3=e$ such that $a-e$ or $a+3e$ is nilpotent and $ae=ea$.

An element $a$ in a ring $R$ is (strongly) clean provided that it is the sum of an idempotent and a unit (that commute). A ring $R$ is (strongly) clean in case every element in $R$ is (strongly) clean. A ring $R$ is an exchange ring provided that for any $a\in R$, there
exists an idempotent $e\in R$ such that $e\in aR$ and $1-e\in (1-a)R$. Every (strongly) clean ring is an exchange ring, but the converse is not true (see~\cite[Proposition 1.8]{N}). A ring $R$ is called a Hirano ring if for any $u\in U(R)$, $1\pm u^2$ is nilpotent. Furthermore, we prove that a ring $R$ is Yaqub nil-clean if and only if $R$ is an exchange Hirano ring. The structure of such rings is thereby completely determined.

We use $N(R)$ to denote the set of all nilpotents in $R$ and $J(R)$ the Jacobson radical of $R$. ${\Bbb N}$ stands for the set of all natural numbers. $a\pm b$ means that $a+b$ or $a-b$. ${\Bbb Z}[u]=\{ f(u)~|~f(t)$ is a polynomial with integral coefficients $\}$.

\section{Elementary Characterizations}

The aim of this section is to investigate elementary characterizations of Yaqub nil-clean rings. We begin with

\begin{lem} Let $R$ be a ring. Then the following are equivalent:
\end{lem}
\begin{enumerate}
\item [(1)] {\it $R$ is Yaqub nil-clean.}
\vspace{-.5mm}
\item [(2)] {\it For any $a\in R$, $a^2\in R$ is strongly weakly nil-clean.}
\end{enumerate}
\begin{proof} $\Longrightarrow$ Let $a\in R$. Then $a\pm a^3\in N(R)$, and so $a^2-a^4$ or $a^2+a^4\in N(R)$. Thus, $a^2-a^4$ or $-a^2-(-a^2)^2$ is nilpotent. That is, $a^2\in R$ is weakly nil-clean.

$\Longleftarrow$ Suppose that $a^2$ is strongly weakly nil-clean. Then $a^2-a^4\in N(R)$ or $-a^2-(-a^2)^2\in N(R)$. This implies that
$a(a-a^3)$ or $a(a+a^3)$ is nilpotent; hence, $(a-a^3)^2$ or $(a+a^3)^2$ is nilpotent. Therefore $a\pm a^3\in N(R)$, as desired.\end{proof}

\begin{thm} Let $R$ be a ring. Then $R$ is Yaqub nil-clean if and only if
\end{thm}
\begin{enumerate}
\item [(1)] {\it $J(R)$ is nil;}
\vspace{-.5mm}
\item [(2)] {\it $R/J(R)$ has the identity $x^3=\pm x$.}
\end{enumerate}
\begin{proof} $\Longrightarrow$ Let $x\in J(R)$. Then $x\pm x^3\in N(R)$; hence, $x\in N(R)$. This shows that $J(R)$ is nil.

Let $a\in R$. Then $a\pm a^3\in N(R)$, and so $a^3\pm a^5\in N(R)$.
Thus, $a-a^5=(a\pm a^3)+(\mp a^3-a^5)\in N(R)$. In light of~\cite[Theorem 2.11]{Z}, $R/J(R)$ has the identity $x^5=x$; hence, it is commutative. We infer that
$N(R)\subseteq J(R)$. This shows that
$\overline{a^3}=\overline{\pm a}$ in $R/J(R)$.

$\Longleftarrow $ Let $a\in R$. Then $a^3\pm a\in J(R)\subseteq N(R)$, as required.\end{proof}

\begin{lem} Every subring of Yaqub nil-clean rings is Yaqub nil-clean.
\end{lem}
\begin{proof} Let $S$ be a subring of a Yaqub nil-clean ring $R$. For any $a\in S$, we see that $a\in R$, and so there exists some $n\in {\Bbb N}$ such that $(a\pm a^3)^n=0$ in $R$; hence, $(a\pm a^3)^n=0$ in $S$. This implies that $S$ is Yaqub nil-clean.\end{proof}

As a consequence of Lemma 2.3, every center of a Yaqub nil-clean ring is Yaqub nil-clean. This generalizes ~\cite[Theorem 2]{ST} as well.

\begin{prop} Every corner of Yaqub nil-clean rings is Yaqub nil-clean.
\end{prop}
\begin{proof} Let $e\in R$ be an idempotent. We will suffice to prove that $eRe$ is Yaqub nil-clean. As $eRe$ is a subring of $R$, we complete the proof by
Lemma 2.3.\end{proof}

\begin{thm} Let
$\{R_{i}~|~i\in I\}$ be a family of rings. Then the direct product $R=\prod\limits_{i\in I}
R_{i}$ of rings $R_i$ is Yaqub nil-clean if and only if each
$R_{i}$ is Yaqub nil-clean and at most one is not strongly 2-nil-clean.\end{thm}
\begin{proof} $\Longrightarrow$ As homomorphic images of $R$, we see that all $R_i$ are Yaqub nil-clean rings. Suppose that
$R_{k}$ and $R_{l} (k\neq l)$ are not strongly 2-nil-clean. Then we can find some $a\in R_k$ such that $a-a^3\not\in N(R_k)$ and $2\not\in N(R_l)$.
Then $(a,1)\in R_k\times R_l$ and $(a,1)-(a,1)^3, (a,1)+(a,1)^3\not\in N(R_k\times R_l)$. Thus, $R_k\times R_l$ is not Yaqub nil-clean. This contradicts to the Yaqub nil-cleanness of $R$. Therefore at most one $R_i$ is not strongly 2-nil-clean.

$\Longleftarrow $ If each $R_i$ is strongly 2-nil-clean, then so is $R$. If $R_k$ is Yaqub nil-clean and each $R_i (i\neq k)$ is strongly 2-nil-clean. One easily checks that $R\cong \big(\prod\limits_{i\neq k}R_i\big)\times R_k$ is Yaqub nil-clean, as asserted.\end{proof}

In particular, we have

\begin{cor} Let
$L=\prod\limits_{i\in I}R_{i}$ be the direct product of rings
$R_i\cong R$ and $|I|\geq 2$. Then $L$ is Yaqub nil-clean if and
only if $R$ is strongly 2-nil-clean if and only if $L$ is strongly 2-nil-clean.\end{cor}

\begin{lem} Let $I$ be a nil ideal of a ring $R$. Then $R$ is Yaqub nil-clean if and only if $R/I$ is Yaqub nil-clean.\end{lem}\begin{proof} One direction is obvious. Conversely, assume that $R/I$ is Yaqub nil-clean. Let $a\in R$. Then $\overline{a\pm a^3}\in N(R/I)$, and so $(a\pm a^3)^{m}\in I$ for some $m\in {\Bbb N}$. As $I$ is nil, we have $n\in {\Bbb N}$ such that $(a\pm -a^3)^{mn}=0$, i.e, $a\pm a^3\in N(R)$. This completes the proof.\end{proof}

We use $T_n(R)$ to denote the ring of all $n\times n$ upper triangular matrices over a ring $R$. We have

\begin{thm} Let $R$ be a ring, and let $n\geq 2$. Then the following are equivalent:
\end{thm}
\begin{enumerate}
\item [(1)] {\it $T_n(R)$ is Yaqub nil-clean.}
\vspace{-.5mm}
\item [(2)] {\it $T_n(R)$ is strongly 2-nil-clean.}
\vspace{-.5mm}
\item [(3)] {\it $R$ is strongly 2-nil-clean.}
\end{enumerate}
\begin{proof} $(1)\Rightarrow (3)$ Let $I=\{
\left(
\begin{array}{cccc}
0&a_{12}&\cdots&a_{1n}\\
&0&\cdots &a_{2n}\\
&&\ddots&\vdots\\
&&&0
\end{array}
\right)\in T_n(R) | $ each $a_{ij}\in R\}$. Then $T_n(R)/I\cong \prod\limits_{i\in I}R_{i}$, where each $R_i=R$. Clearly,
$\prod\limits_{i\in I}R_{i}$ is Yaqub nil-clean. In light of Corollary 2.6, $R$ is strongly 2-nil-clean.

$(3)\Rightarrow (2)$ This is proved in ~\cite[Corollary 2.6]{CS}.

$(2)\Rightarrow (1)$ This is obvious.\end{proof}

\section{Structure Theorems}

\vskip4mm The aim of this section is to investigate structure of Yaqub nil-clean rings. A ring $R$ is periodic if for any $a\in R$ there exist distinct $m,n\in {\Bbb N}$ such that $a^m=a^n$. We now derive

\begin{thm} A ring $R$ is Yaqub nil-clean if and only if $R\cong R_1, R_2,R_3,R_1\times R_2$ or $R_1\times R_3$, where
\end{thm}
\begin{enumerate}
\item [(1)] {\it $R_1/J(R_1)$ is Boolean and $J(R_1)$ is nil;}
\vspace{-.5mm}
\item [(2)] {\it $R_2/J(R_2)$ is a Yaqub ring and $J(R_2)$ is nil.}
\vspace{-.5mm}
\item [(3)] {\it $R_3/J(R_3)\cong {\Bbb Z}_5$ and $J(R_3)$ is nil.}
\end{enumerate}
\begin{proof} $\Longleftarrow $ In view of ~\cite[Theorem 2.7]{KWZ}, $R_1$ is strongly nil-clean. By virtue of~\cite[Theorem 4.2]{CS}, $R_2$ is strongly 2-nil-clean.
Hence, $R_1,R_2$ and $R_1\times R_2$ are strongly 2-nil-clean, and then Yaqub nil-clean. Let $a\in R_3$. Since $R_3/J(R_3)\cong {\Bbb Z}_5$, we easily check that $\overline{a\pm a^3}=\overline{0}$ in $R_3/J(R_3)$; and so $a\pm a^3\in J(R_3)\subseteq N(R_3)$. Thus, $R_3$ is Yaqub nil-clean.
Let $(a,b)\in R_1\times R_3$. Then $b\pm b^3\in N(R_3)$.
Since $R_1/J(R_1)$ is Boolean, we see that $(a,b)\pm (a,b)^3=(a-a^3,b\pm b^3)=(a-a^2+a(a-a^2),b\pm b^3)\in N(R_1\times R_3)$. Hence, $R_1\times R_3$ is Yaqub nil-clean. Therefore $R$ is Yaqub nil-clean.

$\Longrightarrow$ Since $2\pm 2^3\in N(R)$, we see that $2\times 3\in N(R)$ or $2\times 5\in N(R)$. Let $r\in J(R)$. Then $r\pm r^3\in N(R)$, and so $r\in J(R)$. Thus, $J(R)$ is nil.

Case I. $2\times 3\in N(R)$. Then $R\cong R_1,R_2$ or $R_1\times R_2$, where $R_1=R/2^nR,R_2=R/3^nR (n\in {\Bbb N})$.

Case II. $2\times 5\in N(R)$. Then $R\cong R_1,R_3$ or $R_1\times R_3$, where $R_1=R/2^nR,R_3=R/5^nR (n\in {\Bbb N})$.

We easily see that each $J(R_i)$ is nil.

Step 1. Let $S=R_1/J(R_1)$. As $2\in N(S)$, we easily see that $S$ is strongly 2-nil-clean. In view of~\cite[Theorem 2.11]{CS}, $S$ is strongly nil-clean.
According to~\cite[Theorem 2.7]{KWZ}, $S\cong S/J(S)$ is Boolean.

Step 2. Let $S=R_2/J(R_2)$. The $3\in N(S)$. Let $a\in S$. We claim that $a-a^3\in N(S)$. Otherwise, we have $a+a^3\in N(S)$, and so $(1+a)^3-(1+a)$ or $(1+a)^3+(1+a)\in N(S)$. This implies that $a^3-a\in N(S)$ or $a^3+a+2\in N(S)$. This gives a contradiction. Thus, $S$ is strongly 2-nil-clean, by~\cite[Theorem 2.3]{CS}; hence, $S/J(S)$ has the identity $x^3=x$ by ~\cite[Theorem 3.3]{CS}. In light of~\cite[Lemma 4.4]{CS}, $S\cong S/J(S)$ is a Yaqub ring.

Step 3. Let $S=R_3/J(R_3)$. Hence, $x^3\pm x\in N(S)$ for all $x\in S$. Firstly, we claim that $S$ is reduced. If not, there exists $0\neq x\in R$ such that $x^2=0$. As in the proof of~\cite[Proposition 3.5]{CS}, $R$ is periodic; hence, it is clean. Thus, there exists $0\neq g^2=g\in SxS$ such that $gSg\cong M_2(T)$, where $T$ is a nontrivial ring (see~\cite[Lemma 2.7]{Z}). Choose $y=\left(
\begin{array}{cc}
1&1\\
1&0
\end{array}
\right)\in M_2(T)$. Then $y^3=\left(
\begin{array}{cc}
3&2\\
2&1
\end{array}
\right).$ Hence,
$$y-y^3=\left(
\begin{array}{cc}
-2&-1\\
-1&-1
\end{array}
\right)~\mbox{and}~y+y^3=\left(
\begin{array}{cc}
4&3\\
3&0
\end{array}
\right).$$ We easily check that
$$(y-y^3)^2=\left(
\begin{array}{cc}
0&3\\
3&2
\end{array}
\right)~\mbox{and}~y+y^3=\left(
\begin{array}{cc}
0&2\\
2&-1
\end{array}
\right).$$ As $2\in U(T)$, we see that $(y-y^3)^2, (y+y^3)^2$ are invertible, and so $y-y^3,y+y^3\not\in N(M_2(T))$. This gives a contradiction. Therefore $S$ is reduced.

Secondly, we claim that $S$ has no nontrivial idempotents. Clearly, $S$ is abelian. Assume that $0,1\neq g^2=g\in S$. Then $S\cong gS\times (1-g)S$. In view of Theorem 2.5, $gS$ or $(1-g)S$ is strongly 2-nil-clean. Thus, $6\in N(gS)$ or $6\in N((1-g)S)$. But $5\in N(S)$ in $S$, we see that $1\in N(gS)$ or $1\in N((1-g)S)$, a contradiction. Thus, $g=0$ or $1$.

Therefore $S$ is a reduced ring without any trivial idempotents and $5\in N(S)$. Let $0,1\neq u\in R$. Then $u^3=\pm u$, and so $u^2$ or $-u^2$ is an idempotent. This implies that $u^2=1$ or $u^2=-1$.

Case I. $u^2=1$. Then $0,1\neq u-1\in R$. By the preceding discussion, we get $(u-1)^2=1$ or $(u-1)^2=-1$. This shows that
$u=2$ or $-1$. We infer that $u=4$.

Case II. $u^2=-1$. If $u\neq 2$, then $0,1\neq u-1\in R$. Hence, we have $(u-1)^2=1$ or $(u-1)^2=-1$.
This shows that $u=2$ or $3$. This implies that $u=2$ or $3$.

Therefore we conclude that $S={\Bbb Z}_5$.\end{proof}

\begin{cor} A ring $R$ is Yaqub nil-clean if and only if
\end{cor}
\begin{enumerate}
\item [(1)] {\it $a-a^5\in R$ is nilpotent for all $a\in R$;}
\vspace{-.5mm}
\item [(2)] {\it $R$ has no homomorphic images ${\Bbb Z}_3\times {\Bbb Z}_5, {\Bbb Z}_5\times {\Bbb Z}_5$.}
\end{enumerate}
\begin{proof} $\Longrightarrow$ Let $a\in R$. Then $a\pm a^3\in N(R)$, and so $a-a^5=(a\pm a^3)\mp a^2(a\pm a^3)\in N(R)$. By virtue of Theorem 3.1, we easily see that $R$ has no homomorphic images ${\Bbb Z}_3\times {\Bbb Z}_5, {\Bbb Z}_5\times {\Bbb Z}_5$.

$\Longleftarrow $ In light of~\cite[Theorem 2.1]{Z}, $R\cong A,B,C$ or product of such rings, where $A/J(A)$ is Boolean with $J(A)$ is ni, $B/J(B)$ is a subdirect product of ${\Bbb Z}_3's$ with $J(B)$ is nil, and $C/J(C)$ is a subdirect product of ${\Bbb Z}_5's$ with $J(C)$ is nil. By hypothesis,
we prove that $R$ is Yaqub nil-clean, in terms of Theorem 3.1.\end{proof}

\begin{cor} A ring $R$ is strongly 2-nil-clean if and only if
\end{cor}
\begin{enumerate}
\item [(1)] {\it $6\in R$ is nilpotent;}
\vspace{-.5mm}
\item [(2)] {\it $R$ is Yaqub nil-clean.}
\end{enumerate}
\begin{proof} $\Longrightarrow$ In view of~\cite[Theorem 3.6]{CS}, $6\in N(R)$. $(2)$ is obvious.

$\Longleftarrow$ Since $6\in N(R)$, we see that $5\in U(R)$. In view of Theorem 3.1,  $R\cong R_1, R_2$ or $R_1\times R_2$, where
$R_1/J(R_1)$ is Boolean with $J(R_1)$ nil and $R_2/J(R_2)$ is a Yaqub ring with $J(R_2)$ nil. This completes the proof by~\cite[Theorem 4.5]{CS}.\end{proof}

\begin{cor} A ring $R$ is strongly nil-clean if and only if\end{cor}
\begin{enumerate}
\item [(1)] {\it $2$ is nilpotent;}
\vspace{-.5mm}
\item [(2)] {\it $R$ is a Yaqub nil-clean.}
\end{enumerate}
\begin{proof} $\Longrightarrow$ $(1)$ follows from~\cite[Proposition 3.14]{D}.

$(2)$ This is obvious, by ~\cite[Corollary 2.5]{KWZ}.

$\Longleftarrow$ As $R$ is Yaqub nil-clean and $6\in N(R)$, $R$ is strongly 2-nil-clean by Corollary 3.3. Since $2\in N(R)$, it follows by~\cite[Theorem 2.11]{CS} that $R$ is strongly nil-clean.\end{proof}

\begin{exam} Let $R={\Bbb Z_n} (n\geq 2)$. Then $R$ is a Yaqub nil-clean ring if and only if $n=2^k3^l5^s~ (k,l,s~\mbox{are nonnegitive integers and}~ ls=0)$.\end{exam}
\begin{proof} $\Longrightarrow$ Let $p$ be a prime such that $p|n$. Then $n=pq$ with $(p,q)=1$. Hence, $R\cong {\Bbb Z}_p\times {\Bbb Z}_q$.
This shows that ${\Bbb Z}_p$ is a Yaqub nil-clean ring. Hence, $p=2,3$ or $5$. If $kl\neq 0$, then ${\Bbb Z}_{3}\times {\Bbb Z}_5$ is a Yaqub nil-clean, a contradiction. Therefore $n=2^k3^l5^s$ for some nonnegitive integers $k,l,s$ and $ls=0$.

$\Longleftarrow$ Since $n=2^k3^l5^s (ls=0)$, we see that $R\cong {\Bbb Z}_{2^k}\times{\Bbb Z}_{3^l}$ or ${\Bbb Z}_{2^k}\times{\Bbb Z}_{5^l}$. Clearly,
$J({\Bbb Z}_{2^k})=2{\Bbb Z}_{2^k}, J({\Bbb Z}_{3^l})=3{\Bbb Z}_{3^l}$ and $J({\Bbb Z}_{5^s})=5{\Bbb Z}_{5^s}$ are all nil. Moreover, $${\Bbb Z}_{2^k}/J({\Bbb Z}_{2^k})\cong {\Bbb Z}_2, {\Bbb Z}_{3^l}/J({\Bbb Z}_{3^l})\cong {\Bbb Z}_3~\mbox{and}~{\Bbb Z}_{5^s}/J({\Bbb Z}_{5^s})\cong {\Bbb Z}_5.$$ According to
Theorem 3.1, $R$ is a Yaqub nil-clean ring.\end{proof}

We are now ready to prove the following.

\begin{thm} A ring $R$ is Yaqub nil-clean if and only if
\end{thm}
\begin{enumerate}
\item [(1)] {\it $J(R)$ is nil;}
\vspace{-.5mm}
\item [(2)] {\it $R/J(R)$ is isomorphic to a Boolean ring $R_1$, a Yaqub ring $R_2$, ${\Bbb Z}_5$, $R_1\times R_2$, or $R_1\times {\Bbb Z}_5$.}
\end{enumerate}
\begin{proof} $\Longrightarrow $ In view of Theorem 3.1, $R\cong R_1, R_2,R_3, R_1\times R_2$ or $R_1\times R_3$, where
\begin{enumerate}
\item [(i)] {\it $R_1/J(R_1)$ is Boolean and $J(R_1)$ is nil;}
\vspace{-.5mm}
\item [(ii)] {\it $R_2/J(R_2)$ is a Yaqub ring and $J(R_2)$ is nil.}
\vspace{-.5mm}
\item [(iii)] {\it $R_3/J(R_3)\cong {\Bbb Z}_5,$ $J(R_3)$ is nil.}
\end{enumerate} Therefore $J(R)$ is nil and $R/J(R)\cong R_1/J(R_1), R_2/J(R_2), R_3/J(R_3),$ $R_1/J(R_1)\times R_2/J(R_2)$ or $R_1/J(R_1)\times R_3/J(R_3)$, as required.

$\Longleftarrow$ Let $a\in R$. By hypothesis, we easily check that $\overline{a\pm a^3}=\overline{0}$. As $J(R)$ is nil, $a\pm a^3\in J(R)\subseteq N(R)$, as desired.\end{proof}

\begin{cor} A ring $R$ is Yaqub nil-clean if and only if\end{cor}
\begin{enumerate}
\item [(1)]{\it $R$ is periodic;}
\vspace{-.5mm}
\item [(2)]{\it $R/J(R)$ is isomorphic to a Boolean ring $R_1$, a Yaqub ring $R_2$, ${\Bbb Z}_5$, $R_1\times R_2$, or $R_1\times {\Bbb Z}_5$.}
\end{enumerate}
\begin{proof} $\Longrightarrow$ As in the proof of~\cite[Proposition 3.5]{CS}, $R$ is periodic. $(2)$ follows by Theorem 3.6.

$\Longleftarrow$ Since $R$ is periodic, we easily check that $J(R)$ is nil. This completes the proof by Theorem 3.6.\end{proof}

\begin{lem} Let $R$ be a ring with $5\in N(R)$, and let $a\in R$. Then the following are equivalent:\end{lem}
\begin{enumerate}
\item [(1)]{\it $a+a^3\in R$ is nilpotent.}
\vspace{-.5mm}
\item [(2)]{\it There exists $e\in {\Bbb Z}[a]$ such that $a-e\in N(R)$ and $e^3=4e$.}
\end{enumerate}
\begin{proof} $(1)\Rightarrow (2)$ Suppose that $a+a^3\in R$ is nilpotent. Set $x=3a$. Then
$x^3-x=-30a+w$ for some $w\in N(R)$. This shows that $x^3-x\in N(R)$. As $(5^n,2)=1$, we easily see that $2\cdot 1_R\in U(R)$.
In light of~\cite[Lemma 2.6]{KWZ}, there exists $\theta\in {\Bbb Z}[x]$ such that $\theta^3=\theta$ and $x-\theta\in N(R)$.

Take $\beta=2(x-\theta)-5a$. Then $\beta\in N(R)$. Further, we see that $\beta=a-2\theta$.
Set $e=2\theta\in R$. Then $a-e\in N(R)$ and $e\in {\Bbb Z}[a]$. Moreover,
$e^3-4e=8\theta^3-8\theta=0$, as desired.

$(2)\Rightarrow (1)$ Let $a\in R$. Then we have $e\in {\Bbb Z}[a]$ such that $w:=a-e\in N(R)$ and $e^3=4e$. Hence, $a+a^3=(e+w)+(e^3+3e^2w+3ew^2+w^3)=5e+(3e^2+3ew+w^2)w\in N(R)$, as required.\end{proof}

\begin{lem} Let $R$ be a ring with $5\in N(R)$, and let $a\in R$. Then the following are equivalent:\end{lem}
\begin{enumerate}
\item [(1)]{\it $a+a^3\in R$ is nilpotent.}
\vspace{-.5mm}
\item [(2)]{\it There exists $e^3=e\in R$ such that $a+3e\in N(R)$ and $ae=ea$.}
\end{enumerate}
\begin{proof} $(1)\Rightarrow (2)$ In view of Lemma 3.8, there exists $f\in {\Bbb Z}[a]$ such that $a-f\in N(R)$ and
$f^3=4f$. As $5\in R$ is nilpotent, we see that $2\in U(R)$. Set $e=\frac{f}{2}$. Then $e^3=e$ and $a+3e=(a-2e)+5e=(a-f)+5e\in N(R)$, as desired.

$(2)\Rightarrow (1)$ Let $a\in R$. Then we have $e^3=e$ such that $w:=a+3e\in N(R)$ and $ae=ea$. This implies that
$a+a^3=(3e+w)+(27e^3+27e^2w+9ew^2+w^3)=30e+(27e^2+9ew+w^2)w\in N(R)$, as needed.\end{proof}

We come now to the demonstration for which this section has been developed.

\begin{thm} Let
$R$ be a ring. Then the following are equivalent:\end{thm}
\begin{enumerate}
\item [(1)]{\it $R$ is Yaqub nil-clean.}
\vspace{-.5mm}
\item [(2)]{\it For any $a\in R$, there exists $e^3=e$ such that $a-e$ or $a+3e$ is nilpotent and $ae=ea$.}
\end{enumerate}\begin{proof} $(1)\Rightarrow (2)$ In light of Theorem 3.1, $R\cong R_1, R_2,R_3,R_1\times R_2$ or $R_1\times R_3$, where
\begin{enumerate}
\item [(i)] {\it $R_1/J(R_1)$ is Boolean and $J(R_1)$ is nil;}
\vspace{-.5mm}
\item [(ii)] {\it $R_2/J(R_2)$ is a Yaqub ring and $J(R_2)$ is nil.}
\vspace{-.5mm}
\item [(iii)] {\it $R_3/J(R_3)\cong {\Bbb Z}_5$ and $J(R_3)$ is nil.}
\end{enumerate}

Case I. $R\cong R_1,R_2$ or $R_1\times R_2$. By virtue of~\cite[Theorem 4.5]{CS}, $R$ is strongly 2-nil-clean. Then for any $a\in R$, there exists $e^3=e$ such that $a-e$ is nilpotent.

Case II. $R\cong R_1,R_3$ or $R_1\times R_3$. Let $a\in R_1$. As $R_1$ is strongly nil-clean,
there exists an idempotent $e\in R_1$ such that $a-e\in N(R_1)$ and $ae=ea$. Since $2\in N(R_1)$, we see that $a+3e=a-e+4e\in N(R_1)$. Let $a\in R_3$. As $5\in N(R_3)$, we see that $2\in U(R_3)$. Let $a\in R_3$. Then $a-a^3\in N(R_3)$ or $a+a^3\in N(R_3)$. If $a-a^3\in N(R_3)$, by ~\cite[Lemma 2.6]{KWZ}, there exists $e^3=e\in R_3$ such that $a-e\in N(R_3)$ and $ae=ea$. If $a+a^3\in N(R_3)$, it follows by Lemma 3.9 that there exists $e^3=e\in R_3$ such that $a+3e\in N(R_3)$.
Therefore for any $x\in R_1\times R_3$, we can find $f^3=f\in R_1\times R_3$ such that $x-f$ or $x+3f$ is nilpotent in $R_1\times R_3$ and $xf=fx$, as desired.

$(2)\Rightarrow (1)$ By hypothesis, there exists $e^3=e$ such that $2-e$ or $2+3e$ is nilpotent. Hence, $2^3-2$ or $2^3-2\times 9$ is nilpotent.
This shows that $2\times 3\in N(R)$ or $2\times 5\in N(R)$. We infer that $30=2\times 3\times 5\in N(R)$.

Let $a\in R$. Then there exists $f^3=f\in R$ such that $a-f$ or $a+3f$ is nilpotent and $af=fa$. If $w:=a-f\in N(R)$, then $a-a^3=(f+w)-(f+w)^3\in N(R)$.
If $w:=a+3f\in N(R)$, then $a+a^3=(-3f+w)+(-3f+w)^3=-30f+(w^2-18f)w\in N(R)$, and so $a+a^3\in N(R)$. Therefore $R$ is Yaqub nil-clean.\end{proof}

\section{Hirano rings}

The goal of this section is to investigate elementary properties of Hirano rings which will be used in the sequel. We now derive

\begin{prop}
\end{prop}
\begin{enumerate}
\item [(1)] {\it Every subring of Hirano ring is a Hirano ring.}
\vspace{-.5mm}
\item [(2)] {\it If $R$ is a Hirano ring, then $eRe$ is a Hirano ring for all idempotents $e\in R$.}
\end{enumerate}
\begin{proof} $(1)$ Let $S$ be a subring of a Yaqub ring $R$, and $u\in U(S)$, so $u\in U(R)$ and $1_R\pm u^2\in N(R)$, this implies that $\pm u^2=1_R+w$ for some nilpotent element $w\in R$. Thus, $\pm u^2=\pm u^2\times 1_S=1_S+w\times 1_S$. As $u^2$ and $1_S$ are in $S$, then $w\times 1_S\in S$, and therefore $S$ is a Hirano ring.

$(3)$ This is obvious as $eRe$ is a subring of $R$.\end{proof}

We note that the finite direct product of Hirano rings may be not a Hirano ring.

\begin{exam} Let $R={\Bbb Z}_5\times {\Bbb Z}_5$. Then ${\Bbb Z}_5$ is a Hirano ring, while $R$ is not.\end{exam}
\begin{proof} Clearly, ${\Bbb Z}_5$ is a Hirano ring. Choose $u=(1,2)\in R$. Then $u\in U(R)$. We see that $(1,1)+u^2=(2,0)$ and $(1,1)-u^2=(0,2)$; hence,
$1_R+u^2$ and $1_R-u^2$ are not nilpotent. Thus, $R$ is not a Hirano ring.\end{proof}

\begin{exam} Let $R={\Bbb Z}_{5^n}[x]$ is a Hirano ring, but it is not clean.\end{exam}
\begin{proof} Let $f(x)=a_0+a_1x+\cdots +a_nx^n\in U(R)$. Then $5\nmid a_0$ and $5|a_i (i=1,\cdots,a_n)$.
Clearly, ${\Bbb Z}_{5^n}$ is a Hirano ring and $a_0\in U({\Bbb Z}_{5^n})$. Thus, $1\pm a_0\in N({\Bbb Z}_{5^n})$, i.e., $5|(1\pm a_0)$.
This shows that $5|(1\pm f(x))$, and so $1\pm f(x)\in N(R)$. Therefore $R$ is a Hirano ring. But it is not clean, as $x\in R$ can not be written as the
sum of an idempotent and a unit in $R$.
\end{proof}

\begin{lem} Let $I$ be a nil ideal of a ring $R$. Then $R$ is a Hirano ring if and only if so is $R/I$.\end{lem}
\begin{proof} $\Longrightarrow$ This is obvious.

$\Longleftarrow$ Let $u\in U(R)$, so $\pm \bar{u}^2=\bar{1}+\bar{w}$ for $\bar{w}\in N(R/I)$. Hence, $\pm u^2= 1+w+r$ for some $r\in I$. Here $w+r\in N(R)$.
This yields the result.\end{proof}

Recall that a ring $R$ is a 2-UU ring if for any $u\in U(R)$, $u^2$ is a unipotent, i.e., $1-u^2\in N(R)$~\cite{CS2}. We now derive

\begin{lem} Let
$L=\prod\limits_{i\in I}R_{i}$ be the direct product of rings
$R_i\cong R$ and $|I|\geq 2$. Then $L$ is a Hirano ring if and
only if $R$ is a 2-UU ring if and only if $L$ is a 2-UU ring.\end{lem}
\begin{proof} In view of~\cite[Theorem 2.1]{CS2}, $R$ is a 2-UU ring if and only if $L$ is a 2-UU ring. If $L$ is a 2-UU ring, we easily see that $L$ is a Hirano ring.

Suppose that $L$ is a Hirano ring. Then $R$ is a Hirano ring as a subring of $L$. If $R$ is not a 2-UU ring,
we can find some $u\in U(R)$ such that $u^2-1\not\in N(R)$. Additionally, $2\not\in N(R)$.
Choose $v:=(u,1,1,\cdots )\in U(L)$. Then $v^2-1_L, v^2+1_L\not\in N(L)$. This implies that
$L$ is not a Hirano ring, a contradiction. Therefore $R$ is a 2-UU ring, as asserted.\end{proof}

\begin{thm} Let $R$ be a ring, and let $n\geq 2$. Then the following are equivalent:
\end{thm}
\begin{enumerate}
\item [(1)] {\it $T_n(R)$ is a Hirano ring.}
\vspace{-.5mm}
\item [(2)] {\it $T_n(R)$ is a 2-UU ring.}
\vspace{-.5mm}
\item [(3)] {\it $R$ is a 2-UU ring.}
\end{enumerate}
\begin{proof} $(1)\Rightarrow (3)$ Choose $I$ as in the proof of Theorem 2.8. Then $I$ is a nil ideal of $R$. As $T_n(R)/I\cong \prod\limits_{i=1}^{n}R_{i}$ be the direct product of rings
$R_i\cong R$, it follows by Lemma 4.4 that $\prod\limits_{i=1}^{n}R_{i}$ is a Hirano ring. In light of Lemma 4.5, $R$ is a 2-UU ring, as required.

$(3)\Rightarrow (2)$ This is proved in~~\cite[Theorem 2.1]{CS2}.

$(2)\Rightarrow (1)$ This is trivial.\end{proof}

\begin{exam} The ring $M_2({\Bbb Z}_2)$ is not a Hirano ring.\end{exam}
\begin{proof} Choose $U=\left(
\begin{array}{cc}
0&1\\
1&1\\
\end{array}
\right)$. As $I_2\pm U^2=I_2\pm \left(
\begin{array}{cc}
1&1\\
1&2\\
\end{array}
\right)=\left(
\begin{array}{cc}
0&1\\
1&1\\
\end{array}
\right), \left(
\begin{array}{cc}
0&1\\
1&1\\
\end{array}
\right)$, we see that $I_2+U^2$ and $I_2-U^2$ are not nilpotent, as required.\end{proof}

\section{Exchange Properties}

The class of exchange rings is very large. For instances, local rings, regular rings, $\pi$-regular rings, (strongly) clean rings and $C^*$-algebras with real rank one are all exchange rings. We now characterize Yaqub nil-clean rings by means of their exchange properties.

\begin{lem} Let $R$ be an exchange ring. Then $-2\in R$ is clean.\end{lem}
\begin{proof} See~\cite[Lemma 4.2]{CS2}.\end{proof}

\begin{lem} Let $R$ be an exchange Hirano ring. Then $30\in R$ is nilpotent.\end{lem}
\begin{proof} In view of Lemma 5.1, $-2\in R$ is clean. Then $-2=e+u$ for some idempotent $e$ and unit $u$. As $R$ is a Hirano ring, $1\pm u^2=w$ for some $w\in N(R)$.

Case I: $1-u^2=w$. Sine $-1-e=u+1$, then $1+3e=u^2+1+2u$, this implies that $3e-2u=1-w=1+(-w)=1+v$ for some $v\in N(R)$. Hence $3e-2(-2-e)=1+v$, and so $5e=-3+v$. We see that $5(-2-u)=-3+v$, i.e., $5u=-7+v$, and then $25u^2=49+v^2-14v=49+v_1$ for some $v_1\in N(R)$. Thus $24\in N(R)$, and so $6\in N(R)$ which implies $30\in N(R)$.

Case II: $1+u^2=w$. As $-2=e+u$, then $1+3e=u^2+2u+1=w+2u$, so $3e-2u=w-1$. Thus, $2u-3e=1+(-w)=1+w^{\prime}$, $2(-2-e)-3e=1+w^{\prime}$, i.e., $-5-5e=w^{\prime}$. This implies that $-5e=w^{\prime}+5$, and then $-5(-2-u)=w^{\prime}+5$. Hence, $5u=w{\prime}-5$, so $25u^2=25+w^{''}$, which implies that $25(w-1)=25+w{''}$. We infer that $50\in N(R)$, whence $2\times 5\times 5\in N[R]$. Accordingly, $2\times 5\in N(R)$, and therefore $30\in N(R)$.\end{proof}

\begin{lem} Let $R$ be an exchange Hirano ring. Then $J(R)$ is nil.\end{lem}
\begin{proof} In view of Lemma 5.2, $30\in N(R)$. Write $30^n=0 (n\in {\Bbb N})$. Then we can write $R=R_1\times R_2\times R_3$, where $R_1\cong R/2^nR, R_2\cong R/3^nR$ and $R_3\cong R/5^nR$. As $R$ is a Hirano ring, so is $R_1$ by Proposition 4.1, Then for any $u\in U(R_1), 1\pm u^2\in N(R_1)$, also $2\in N(R_1)$. If $1+u^2\in N(R_1)$ we can write $(u-1)^2 =1+u^2-2u\in N(R_1)$ and so $1-u\in N(R_1)$, which implies $R_1$ is a UU ring. As in~\cite[Theorem 2.4]{Da}, $J(R_1)$ is nil. If $1-u^2\in N(R_1)$, then $-(1-u)^2=-u^2-1+2u=1-u^2-2(1-u)\in N(R_1)$, then $1-u\in N(R_1)$ and so $J(R_1)$ is nil. Let $x\in J(R_2)$, as $R_2$ is a Hirano ring, $\pm (1+x)^2=1+w$ for some $w\in N(R_2)$, hence $x(x+2)$ or $x(x+2)+2$ is nilpotent. Case I. $w:=x(x+2)\in N(R)$.
As $3\in N(R_2)$, we see that $2\in U(R_2)$, and so $x+2=2^{-1}(1+2x)\in U(R_2)$. We infer that $x=(x+2)^{-1}w\in N(R_2)$.
Case II. $w:=x(x+2)+2\in N(R)$. Then $x(x+2)=w-2\in U(R_2)$, and so $x\in U(R_2)$, a contradiction. This imples that $J(R_2)$ is nil.
For $R_3$, as $5\in N(R_3)$, we deduce that $2\in U(R_3)$. Thus, by the similar route for $R_2$, we see that $J(R_3)$ is nil. Therefore $J(R)$ is nil, as asserted.
\end{proof}

We have accumulated all the information necessary to prove the following.

\begin{thm} A ring $R$ is Yaqub nil-clean if and only if
\end{thm}
\begin{enumerate}
\item [(1)] {\it $R$ is an exchange ring;}
\vspace{-.5mm}
\item [(2)] {\it $R$ is a Hirano ring.}
\end{enumerate}
\begin{proof} $\Longrightarrow$ By Corollary 3.7, $R$ is periodic, and so it is an exchange ring. Let $u\in U(R)$. Then $u\pm u^3\in N(R)$; hence, $1\pm u^2\in N(R)$. Therefore $R$ is a Hirano ring.

$\Longleftarrow$ Let $0 \neq x\in N(R)$, we can assume that $x^2=0$. As $R$ is an exchange ring with $J(R)=0$, by~\cite[Lemma 2.7]{Z}, we can find some idempotent $e\in R$ and some ring $T$, such that $eRe\cong M_2(T)$, but as we see in Example 4.5, $M_2(T)$ is not a Hirano ring, i.e, $eRe$ is not a Hirano ring. This shows that $R$ is not a Hirano by Proposition 4.1, a contradiction. So we deduce that $N(R)=0$, and then $R$ is a reduced ring. This implies that
$R$ is abelian. Since $R$ is an exchange ring, it follows by~\cite[Proposition 1.8]{N} that $R$ is clean.

In light of Lemma 5.2, $30\in N(R)$. Write $2^n\times 3^n\times 5^n=0 (n\in {\Bbb N}$. Then $R\cong R_1, R_2, R_3$ or products of these rings, where $R_1=R/2^nR,R_2=R/3^nR$ and $R_3=R/5^nR$.

Case 1. $2\in N(R_1)$. Let $a\in R_1$. Then we have a central idempotent $e\in R$ and a unit $u\in R$ such that $a=e+u$.
As $1\pm u^2\in N(R_1)$, we see that $u\in 1+N(R_1)$. Hence, $a^2=e+2eu+u^2$, and so $a-a^2\in N(R_1)$. This implies that $a-a^3=(a-a^2)+a(a-a^2)\in N(R_1)$, and so $R_1$ is Yaqub nil-clean.

Case 2. $3\in N(R_2)$. Let $a\in R_2$. Then we have a central idempotent $e\in R$ and a unit $u\in R$ such that $a=e+u$.
Hence, $a^3=(e+u)^3=e+3eu+3eu^2+u^3$. If $1+u^2\in N(R_2)$, then $u+u^3\in N(R_2)$, and so $a+a^3\in N(R_2)$.
If $-1+u^2\in N(R_2)$, then $u-u^3\in N(R_2)$. Therefore $a-a^3\in N(R_3)$. In any case, $a\pm a^3\in N(R_2)$. This means that $R_2$ is Yaqub nil-clean.

Case 3. $5\in N(R_3)$. Let $a\in R_3$. Then we have a central idempotent $e\in R_3$ and a unit $u\in R_3$ such that $a=e+u$. Then $1\pm u^2\in N(R_3)$, and so $u-u^5\in N(R_3)$. Further, $a^5=(e+u)^5=e^5+5eu+10u^2+10eu^3++5eu^4+u^5$, whence, $a-a^5\in N(R_3)$. Choose $u$ is $(1,2)$ in ${\Bbb Z}_3\times {\Bbb Z}_5$ or
${\Bbb Z}_5\times {\Bbb Z}_5$. Then $1\pm u^2$ is not nilpotent. This implies that $R_3$ has no homomorphic images ${\Bbb Z}_3\times {\Bbb Z}_5$ and ${\Bbb Z}_5\times {\Bbb Z}_5$. According to Corollary 3.2, $R_3$ is Yaqub nil-clean.

Case 4. $R\cong R_1\times R_2, R_1\times R_3$. One easily checks that $R$ is Yaqub nil-clean.

Case 5. $R\cong R_2\times R_3, R_1\times R_2\times R_3$. But $R_2\times R_3$ is not a Hirano ring, as $(1,2)\in U(R_2\times R_3)$ and $(1,1)\pm (1,2)^2\not\in N(R_2\times R_3)$. Thus, this case can not appear.

Therefore $R$ is Yaqub nil-clean.\end{proof}

\begin{cor} A ring $R$ is Yaqub nil-clean if and only if
\end{cor}
\begin{enumerate}
\item [(1)] {\it $R$ is periodic;}
\vspace{-.5mm}
\item [(2)] {\it $R$ is a Hirano ring.}
\end{enumerate}
\begin{proof} $\Longrightarrow$ $(1)$ follows from Corollary 3.7 and $(2)$ is obtained by Theorem 5.4.\\
$\Longleftarrow$  As every periodic ring is an exchange ring then we get the result by Theorem 5.4.\end{proof}

\end{document}